\documentclass[letterpaper,leqno,11pt,oneside]{amsart}

\usepackage{amsmath,amsthm,amsfonts,amssymb,amscd}
\usepackage{enumitem}
\usepackage{mathtools}
\usepackage[extension=pdf]{hyperref}
\usepackage[height=9.6in,width=5.95in]{geometry}
\usepackage{graphics,graphicx}
\usepackage[final,color,notref,notcite]{showkeys}
\usepackage{amsfonts,amsmath, amssymb,amsthm,amscd}
\usepackage{bm}
\usepackage{mathrsfs}
\usepackage{yfonts}
\usepackage[mpexclude,DIV13]{typearea}
\usepackage{verbatim}
\usepackage{hyperref}
\usepackage{graphicx}
\usepackage[latin1]{inputenc}
\usepackage{latexsym}
\usepackage{lscape}
\usepackage{epsfig}
\usepackage{subfigure}

\definecolor{darkblue}{rgb}{0.13,0.13,0.39}
\hypersetup{colorlinks=true,urlcolor=darkblue,citecolor=darkblue,linkcolor=darkblue,pdftitle={Fluctuation statistics for log-Gamma polymer}}

\mathtoolsset{showonlyrefs,showmanualtags}

\newtheorem{thm}{Theorem}
\newtheorem{lem}{Lemma}[section]
\newtheorem{prop}[lem]{Proposition}
\newtheorem{cor}[lem]{Corollary}

\theoremstyle{definition}
\newtheorem{rem}[lem]{Remark}
\newtheorem*{rem*}{Remark}
\newtheorem{defn}[lem]{Definition}

\newcommand{\lc}{\ell}

\newcommand{\I}{{\rm i}}

\newcommand{\pp}{\mathbb{P}}

\newcommand{\ee}{\mathbb{E}}
\newcommand{\rr}{\mathbb{R}}

\newcommand{\cc}{{\mathbb{C}}}

\newcommand{\uno}[1]{\mathbf{1}_{#1}}

\newcommand{\wt}{\widetilde}

\newcommand{\qand}{\quad\text{and}\quad}
\newcommand{\qqand}{\qquad\text{and}\qquad}

\DeclareMathOperator{\im}{Im}
\DeclareMathOperator{\var}{var}

\newcommand{\e}[0]{\epsilon}
\newcommand{\EE}{\ensuremath{\mathbb{E}}}

\newcommand{\PP}{\ensuremath{\mathbb{P}}}

\newcommand{\R}{\ensuremath{\mathbb{R}}}

\newcommand{\C}{\ensuremath{\mathbb{C}}}
\newcommand{\Z}{\ensuremath{\mathbb{Z}}}

\newcommand{\Real}{\ensuremath{\mathrm{Re}}}

\newcommand{\bfone}{\ensuremath{\mathbf{1}}}

\newcommand\thetarc{\gamma}

\newcommand{\Res}[1]{\underset{{#1}}{\mathrm{Res}}}

\numberwithin{equation}{section}

\let\oldmarginpar\marginpar
\renewcommand\marginpar[1]{\-\oldmarginpar[\raggedleft\footnotesize #1]%
{\raggedright{\small\textsf{#1}}}}

\begin{document}

\title{Log-Gamma polymer free energy fluctuations via a Fredholm determinant identity}

\author[A. Borodin]{Alexei Borodin}
\address[A. Borodin]{
Massachusetts Institute of Technology,
Department of Mathematics,
77 Massachusetts Avenue, Cambridge, MA 02139-4307, USA}
\email{borodin@math.mit.edu}

\author[I. Corwin]{Ivan Corwin}
\address[I. Corwin]{
Microsoft Research,
New England, 1 Memorial Drive, Cambridge, MA 02142, USA}
\email{ivan.corwin@gmail.com}

\author[D. Remenik]{Daniel Remenik}
\address[D. Remenik]{
  Department of Mathematics\\
  University of Toronto\\
  40 St. George Street\\
  Toronto, Ontario\\
  Canada M5S 2E4 \newline \indent\textup{and}\indent
  Departamento de Ingenier\'ia Matem\'atica\\
  Universidad de Chile\\
  Av. Blanco Encala\-da 2120\\
  Santiago\\
  Chile} \email{dremenik@math.toronto.edu}

\begin{abstract}
  We prove that under $n^{1/3}$ scaling, the limiting distribution as $n\to \infty$ of the
  free energy of Sepp\"{a}l\"{a}inen's log-Gamma discrete directed polymer is GUE
  Tracy-Widom. The main technical innovation we provide is a general identity between a
  class of $n$-fold contour integrals and a class of Fredholm determinants. Applying this
  identity to the integral formula proved in \cite{cosz} for the Laplace transform of the
  log-Gamma polymer partition function, we arrive at a Fredholm determinant which lends
  itself to asymptotic analysis (and thus yields the free energy limit theorem). The Fredholm determinant was anticipated in \cite{borCor} via the formalism of Macdonald processes yet its rigorous proof was so far lacking because of the nontriviality of certain decay estimates required by that approach.
\end{abstract}

\maketitle

\section{Introduction and main results}

The log-Gamma polymer was introduced and studied by Sepp\"al\"ainen \cite{seppPolymerBoundary}.

\begin{defn} Let $\theta$ be a positive real.
A random variable $X$ has {\it inverse-Gamma distribution with parameter $\theta>0$}  if it is supported on the positive reals
where it has distribution
\begin{equation}\label{invgammadensity}
\PP(X\in dx) = \frac{1}{\Gamma(\theta)} x^{-\theta-1}\exp\left\{-\frac{1}{x}\right\} dx.
\end{equation}
We abbreviate this $X\sim \Gamma^{-1}(\theta)$.
\end{defn}

\begin{defn}
The {\it log-Gamma polymer partition function} with parameter $\gamma>0$ is given by
\begin{equation}
Z(n,N) = \sum_{\pi:(1,1)\to (n,N)} \prod_{(i,j)\in \pi} d_{i,j}
\end{equation}
where $\pi$ is an up/right directed lattice path from the Euclidean point $(1,1)$ to $(n,N)$ and where $d_{i,j}\sim \Gamma^{-1}(\gamma)$.
\end{defn}

In \cite{seppPolymerBoundary} it was proved that
\begin{equation}
\lim_{n\to \infty} \frac{\log Z(n,n)}{n} = \bar{f}_{\gamma}, \qquad \limsup_{n\to \infty} \frac{\var \log Z(n,n)}{n^{2/3}} \leq C
\end{equation}
where $\bar{f}_{\gamma} = -2\Psi(\gamma/2)$ and $C$ is a large constant. Here
$\Psi(x)=[\log \Gamma]'(x)$ is the digamma function. The scale of the variance upper-bound is believed
to be tight, since directed polymers at positive temperature should have KPZ universality
class scalings (see e.g. the review \cite{corwinReview}). Moreover, it is believed that, when
centered by $n \bar{f}_{\gamma}$ and scaled by $n^{1/3}$, the distribution of the free
energy $\log Z(n,n)$ should limit to the GUE Tracy-Widom distribution \cite{tracyWidom}.

We presently prove this form of KPZ universality for the log-Gamma polymer for
$\gamma<\gamma^*$ for some $\gamma^*>0$. This assumption is purely technical and comes
from the asymptotic analysis. It is likely that this assumption can be removed following
the approach of \cite{borodinCorwinFerrari}, where a similar assumption was removed in the
case of the semi-discrete polymer. For this model $\gamma$ plays a role akin to
temperature.

\begin{thm}\label{thm:fluct}
There exists $\gamma^*>0$ such that the log-Gamma polymer free energy with parameter $\gamma\in (0,\gamma^*)$ has limiting fluctuation distribution given by
\begin{equation}
\lim_{n\to \infty} \PP\left(\frac{ \log Z(n,n) - n\bar{f}_{\gamma}}{n^{1/3}}\leq r\right) = F_{{\rm GUE}}\left(\left(\frac{\bar{g}_{\gamma}}{2}\right)^{-1/3} r\right)
\end{equation}
where $\bar{f}_{\gamma} = -2\Psi(\gamma/2)$, $\bar{g}_{\gamma} = -2 \Psi''(\gamma/2)$ and $F_{{\rm GUE}}$ is the GUE Tracy-Widom distribution function.
\end{thm}

We give the proof of this theorem in Section \ref{sec:Asym}. There are two ingredients in
the proof, and then some asymptotic analysis. The first ingredient is the $n$-fold
integral formula given in \cite{cosz} for the Laplace transform of the polymer partition
function. This is given below as Proposition \ref{COSZform}. The
second ingredient in the proof is a general identity between a class of $n$-fold contour
integrals and a class of Fredholm determinants. This is given below as Theorem
\ref{thm:equiv}. Applying this identity to Proposition \ref{COSZform} yields Corollary
\ref{cor:COSZdet} which is a new Fredholm determinant expression for the Laplace transform
of the log-Gamma polymer partition function. This formula lends itself to straightforward
asymptotic analysis, as is done in Section \ref{sec:Asym}.

The log-Gamma polymer may be generalized, as done in \cite{cosz}, so that the distributions of the $\thetarc_{i,j}$
depend on two collections of parameters. 

Discrete directed polymer partition functions, under {\it intermediate disorder} scaling
\cite{AKQ,QRMF}, converge to the solution of the multiplicative stochastic heat equation
(whose logarithm is the KPZ equation). If the two collections of parameters determining
the distributions of the $\thetarc_{i,j}$ are tuned correctly, then the initial data for
the limiting stochastic heat equation is determined by two collections of parameters as
well. The Fredholm determinant formula of Corollary \ref{cor:COSZdet} should limit to an
analogous formula for the Laplace transform of the stochastic heat equation with this
general class of initial data which would be a finite temperature analog of the results of \cite{BP07}. When only one of the collections of parameters is tuned,
this formula was computed in \cite{borodinCorwinFerrari} via a similar limit of the
Fredholm determinant formula for the Laplace transform of the semi-discrete polymer
partition function (see also \cite{CQ} in the case where only a single parameter is
tuned), and this is a finite temperature analog of the results of \cite{BBP}.


\begin{defn}
The {\it Sklyanin measure} $s_N$ on $\cc^N$ is given by
\begin{equation}\label{eq:skly}
  s_N(dw_1,\dotsc,dw_N)=\frac1{(2\pi\I)^NN!}\prod_{\substack{i,j=1\\i\neq j}}^N
  \frac1{\Gamma(w_i-w_j)} \prod_{i=1}^{N} dw_i.
\end{equation}
\end{defn}

The following result is taken from \cite{cosz}, Theorem 3.8.ii.
\begin{prop}\label{COSZform}
  Fix $n\geq N$, and choose parameters $\alpha_i>0$ for $1\leq i\leq n$ and $a_j>0$ for
  $1\leq i\leq N$ such that $\thetarc_{i,j}=\alpha_i-a_j>0$. Consider the log-Gamma
  polymer partition function where $d_{i,j}\sim \Gamma^{-1}(\thetarc_{i,j})$. Then for all
  $u$ with $\Real(u)>0$
\begin{equation}
\EE\left[ e^{-u Z(n,N)}\right] = \int_{(\I \R)^N} s_N(dw_1,\ldots, dw_N) \prod_{i,j=1}^{N} \Gamma(a_j-w_i) \prod_{j=1}^{N} \frac{F(w_j)}{F(a_j)},
\end{equation}
where
\begin{equation}\label{eqn:Fw}
F(w) = u^w \prod_{m=1}^{n} \Gamma(\alpha_m - w).
\end{equation}
\end{prop}

Until now, there has not been progress in extracting asymptotics from this formula. The
following theorem, however, transforms this integral formula into a Fredholm determinant
for which we can readily perform asymptotic analysis. This identity should be considered
the main technical contribution of this paper, of which Theorem \ref{thm:fluct} is
essentially a corollary (after some asymptotic analysis).

\begin{defn}
  We introduce the following contours: $C_\delta$ is a positively oriented circle around
  the origin with radius $\delta$; $\lc_\delta$ is a line parallel to the imaginary axis
  from $-\I\infty+\delta$ to $\I\infty+\delta$; $-\lc_\delta$ is similarly the contour
  from $-\I\infty-\delta$ to $\I\infty-\delta$; and $\lc'_{\delta_1,\delta_2,M}$ is the
  horizontal line segment going from $-\delta_1+\I M$ to $\delta_2+\I M$.
  For any simple smooth contour $\gamma$ in $\cc$ we will write $L^2(\gamma)$ to mean the
  space $L^2(\gamma,\mu)$ where $\mu$ is the path measure along $\gamma$ divided by $2\pi\I$.
\end{defn}


\begin{thm}\label{thm:equiv}
  Fix $0<\delta_2<1$, $0<\delta_1<\min\{\delta_2,1-\delta_2\}$ and $a_1,\dotsc,a_N\in\cc$ such that
  $|a_i|<\delta_1$. Suppose $F$ is a meromorphic function such that all its poles have
  real part strictly larger than $\delta_2$, $F$ is non-zero along and inside $C_{\delta_1}$, and for all $\kappa>0$
  \begin{equation}
    \int_{\pm \lc_{\delta_2}}dw\,e^{\pi(\frac{N}{2}-1)|\!\im(w)|}|\!\im(w)|^{\kappa}|F(w)|<\infty,
    \quad\int_{\lc'_{\delta_1,\delta_2,M}}\!\!dw\,e^{\pi(\frac{N}{2}-1)|\!\im(w)|}|\!\im(w)|^{\kappa}|F(w)|
    \xrightarrow[|M|\to\infty]{}0.\label{eq:decayAssum}
  \end{equation}
  Then
  \begin{equation}
    \int_{-\lc_{\delta_1}}\!\!\dotsm\!\int_{-\lc_{\delta_1}}
    s_N(dw_1,\dotsc,dw_N)\,\prod_{i,j=1}^N\Gamma(a_j-w_i)\prod_{j=1}^N\frac{F(w_j)}{F(a_j)}
    =\det(I+K)_{L^2(C_{\delta_1})},\label{eq:equiv}
  \end{equation}
  where $\det(I+K)_{L^2(C_{\delta_1})}$ is the Fredholm determinant of
  $K\!:L^2(C_{\delta_1})\longrightarrow L^2(C_{\delta_1})$ with
  \begin{equation}
    K(v,v')=\frac1{2\pi\I}\int_{\lc_{\delta_2}}dw\,\frac\pi{\sin(\pi(v-w))}
    \frac{F(w)}{F(v)}\frac1{w-v'}\prod_{m=1}^N\frac{\Gamma(v-a_m)}{\Gamma(w-a_m)}.\label{eq:defK}
  \end{equation}
\end{thm}

The proof of this theorem is given in Section \ref{Sec:equiv}. The argument uses only the
Andr\'{e}ief identity \cite{And} (sometimes referred to as the generalized Cauchy-Binet
identity), the Cauchy determinant, the fact that $\det(I+AB)=\det(I+BA)$ and some simple
contour shifts and residue computations to go from the Fredholm determinant formula on the
right hand side of \eqref{eq:equiv} to the integral formula on the left hand side. Going
from the integral formula to the Fredholm determinant formula (even knowing that it should
be true) is a much more challenging path, since one has to undo certain cancelations to
discover determinants.

\begin{rem}
  In the context of the related semi-discrete polymer (see the end of the introduction)
  O'Connell describes (after Corollary 4.2 of \cite{oconnellQTL}) another Fredholm
  determinant expression for the Laplace transform of the partition function. That
  expression and the one which arises from the application of Theorem \ref{thm:equiv} are
  different.
\end{rem}

\begin{rem}
  The condition in Theorem \ref{thm:equiv} on the $|a_i|<\delta_1$ can be relaxed at the
  cost of more complicated choices of contours. Instead of taking the $w$ contour to be a
  vertical line $\lc_{\delta_2}$ one could, in order to accommodate larger values of the
  $a_i$, shift the vertical line horizontally to the right by some positive integer $K$,
  and then augment the $w$ contour with a collection of sufficiently small contours around
  the positive integers up to and including $K$. See Section 5 of
  \cite{borodinCorwinFerrari} for an example of this sort of procedure.
\end{rem}

We may apply Theorem \ref{thm:equiv} to show the following.

\begin{cor}\label{cor:COSZdet}
  Fix $n\geq N$ and any $\delta_1,\delta_2$ such that $0<\delta_2<1$ and
  $0<\delta_1<\min\{\delta_2,1-\delta_2\}$. Given a collection of $\alpha_i>\delta_2$ for
  $1\leq i\leq n$, and $0\leq a_j<\delta_1$ for $1\leq j\leq N$, set
  $\thetarc_{i,j}=\alpha_i-a_j$ and consider the log-Gamma polymer partition function with
  weights $d_{i,j}\sim \Gamma^{-1}(\thetarc_{i,j})$. Then for all $u$ with $\Real(u) >0$
  \begin{equation}
    \EE\!\left[ e^{-u Z(n,N)}\right] = \det(I+K_u)_{L^2(C_{\delta_1})},
  \end{equation}
  where $\det(I+K)_{L^2(C_{\delta_1})}$ is the Fredholm determinant of $K\!:L^2(C_{\delta_1})\longrightarrow L^2(C_{\delta_1})$ with
  \begin{equation}
    K_u(v,v')=\frac1{2\pi\I}\int_{\lc_{\delta_2}}dw\,\frac\pi{\sin(\pi(v-w))}
    \frac{F(w)}{F(v)}\frac1{w-v'}\prod_{m=1}^N\frac{\Gamma(v-a_m)}{\Gamma(w-a_m)}
  \end{equation}
  and $F(w)$ as given in \eqref{eqn:Fw}.
\end{cor}

\begin{proof}
  We start with the Laplace transform formula given in Proposition \ref{COSZform} (if some
  $a_j=0$ we simply shift the $\I \R$ integration contour slightly to the left). In order
  to apply Theorem \ref{thm:equiv} we must check that all the conditions are
  satisfied. The function $F(w)$ has no poles with real part less than $\min\{\alpha_i\}$
  and it is non-zero in the entire complex plane. This implies that given $\delta_1$ and
  $\delta_2$ as specified in the hypothesis, the conditions on the poles and zeros of $F$
  are satisfied. It remains to check the decay condition \eqref{eq:decayAssum}. This,
  however, is immediate from the estimate on the Gamma function as given in
  \eqref{eq:gammaBd}. Note that the condition $n\geq N$ becomes important in this case
  (actually $n\geq N-1$ would do). The conditions having been satisfied, we may apply
  Theorem \ref{thm:equiv} to arrive at the corollary.
\end{proof}

The asymptotic analysis of this Fredholm determinant is performed in Section \ref{sec:Asym} and yields the proof of Theorem \ref{thm:fluct}.

The existence of such an identity as in Theorem \ref{thm:fluct} did not arise out of the
blue. Let us briefly explain the two results which suggested this identity (though only in the special case of $F(w)$ as in \eqref{eq:Fw2}).

\begin{defn}
  An {\it up/right path} in $\R\times \Z$ is an increasing path which either proceeds to
  the right or jumps up by one unit. For each sequence $0<s_1<\cdots<s_{N-1}<t$ we can
  associate an up/right path $\phi$ from $(0,1)$ to $(t,N)$ which jumps between the
  points $(s_i,i)$ and $(s_{i},i+1)$, for $i=1,\ldots, N-1$, and is continuous
  otherwise. Fix a real vector $a=(a_1,\ldots, a_N)$ and let $B(s) = (B_1(s),\ldots,
  B_N(s))$ for $s\geq 0$ be independent standard Brownian motions such that $B_i$ has
  drift $a_i$.

Define the {\it energy} of a path $\phi$ to be
\begin{equation*}
E(\phi) = B_1(s_1)+\left(B_2(s_2)-B_2(s_1)\right)+ \cdots + \left(B_N(t) - B_{N}(s_{N-1})\right).
\end{equation*}
Then the {\it O'Connell-Yor semi-discrete directed polymer partition function} $Z^{N}(t)$ is given by
\begin{equation}\label{Zsd}
Z^{N}(t) = \int d\phi\,e^{E(\phi)},
\end{equation}
where the integral is with respect to Lebesgue measure on the Euclidean set of all
up/right paths $\phi$ (i.e., the simplex of jumping times $0<s_1<\cdots<s_{N-1}<t$).
\end{defn}

Due to the invariance principle, the semi-discrete polymer is a universal scaling limit for discrete directed polymers when $N$ is fixed and $n$ goes to infinity (and temperature is suitably scaled -- see for instance \cite{ABC}). As such, the semi-discrete polymer inherits the solvability of the log-Gamma polymer.

In fact, before the work of \cite{cosz} on the log-Gamma polymer, O'Connell \cite{oconnellQTL} showed (see Corollary 4.2) that $\EE\!\left[ e^{-u Z^N(t)}\right]$ was given by the left hand side of \eqref{eq:equiv} with
\begin{equation}
F(w)= u^w e^{w^2t/2}.\label{eq:Fw2}
\end{equation}

On the other hand, the theory of Macdonald processes was developed in \cite{borCor}. Using
Macdonald difference operators and the Cauchy identity for Macdonald polynomials
\cite{borCor} computes expectations of certain observables of the Macdonald processes,
which, when formed into generating functions, lead to Fredholm determinants. After a
particular limit transition, the Macdonald processes become the Whittaker processes
introduced in \cite{oconnellQTL} and the generating functions become a Fredholm
determinant expression for $\EE\!\left[ e^{-u Z^N(t)}\right]$ -- exactly in the form of
the one on the right hand side of \eqref{eq:equiv}. In particular the following result
appeared in \cite{borCor} as Theorem 5.2.10 (a change of variables $s+v=w$ brings it
exactly to the form of \eqref{eq:defK}).
\begin{thm}
Fix $N\geq 1$ and a drift vector $a=(a_1,\ldots,a_N)$. Fix $0<\delta_2<1$, and $\delta_1<\delta_2 /2$ such that $|a_i|<\delta_1$. Then for $t\geq 0$,
\begin{equation*}
\EE\left[ e^{-u Z^{N}(t)}\right] = \det(I+ K_{u})
\end{equation*}
where $\det(I+ K_{u})$ is the Fredholm determinant of
\begin{equation*}
K_{u}: L^2(C_{a})\to L^2(C_{a})
\end{equation*}
for $C_{a}$ a positively oriented contour containing $a_1,\ldots, a_N$ and such that for all $v,v'\in C_{a}$, we have $|v-v'|<\delta_2$. The operator $K_u$ is defined in terms of its integral kernel
\begin{equation*}
K_{u}(v,v') = \frac{1}{2\pi \I}\int_{\lc_{\delta_2}}ds\,\Gamma(-s)\Gamma(1+s) \prod_{m=1}^{N}\frac{\Gamma(v-a_m)}{\Gamma(s+v-a_m)} \frac{ u^s e^{vt s+t s^2/2}}{v+s-v'}.
\end{equation*}
\end{thm}

This provides a very indirect proof of the identity given in \eqref{eq:equiv}, in the very
particular case of $F(w)$ as in \eqref{eq:Fw2}. An obvious question this development
raised was to provide a direct proof of this identity and to understand how general it
is. This is what Theorem \ref{thm:equiv} accomplishes.

It is worth noting that Macdonald processes also have a limit transition to the Whittaker
processes described in \cite{cosz}, which are connected to the log-Gamma polymer.  The
Fredholm determinant given in Corollary \ref{cor:COSZdet} was anticipated in \cite{borCor}
via the formalism of Macdonald processes yet its rigorous proof was so far lacking because
of the nontriviality of certain decay estimates required by that approach. If one only
cares about the GUE Tracy-Widom asymptotics then the approach given here provides a
direct, though non-obvious and rather ad hoc, route from the integral formula of
\cite{cosz} to the Fredholm determinant.

\subsection*{Acknowledgements}

AB was partially supported by the NSF grant DMS-1056390. IC was partially supported by the
NSF through DMS-1208998 as well as by the Clay Research Fellowship and by Microsoft
Research through the Schramm Memorial Fellowship. DR was partially supported by the
Natural Science and Engineering Research Council of Canada, by a Fields-Ontario
Postdoctoral Fellowship and by Fondecyt Grant 1120309. DR is appreciative for MIT's hospitality during the visit in which this project was initiated.

\section{Fredholm determinant asymptotic analysis: Proof of Theorem \ref{thm:fluct}}\label{sec:Asym}

Let us first recall two useful lemmas. In performing steepest descent analysis on Fredholm
determinants, the following allows us to deform contours to descent curves.

\begin{lem}[Proposition 1 of \cite{TW3}]\label{TWprop1}
Suppose $s\to \Gamma_s$ is a deformation of closed curves and a kernel $L(\eta,\eta')$ is analytic in a neighborhood of $\Gamma_s\times \Gamma_s\subset \C^2$ for each $s$. Then the Fredholm determinant of $L$ acting on $\Gamma_s$ is independent of $s$.
\end{lem}

\begin{lem}[Lemma 4.1.38 of \cite{borCor}]\label{problemma1}
Consider a sequence of functions $\{f_n\}_{n\geq 1}$ mapping $\R\to [0,1]$ such that for each $n$, $f_n(x)$ is strictly decreasing in $x$ with a limit of $1$ at $x=-\infty$ and $0$ at $x=\infty$, and for each $\delta>0$, on $\R\setminus[-\delta,\delta]$ $f_n$ converges uniformly to $\bfone(x\leq 0)$. Define the $r$-shift of $f_n$ as $f^r_n(x) = f_n(x-r)$. Consider a sequence of random variables $X_n$ such that for each $r\in \R$,
\begin{equation*}
\EE[f^r_n(X_n)] \to p(r)
\end{equation*}
and assume that $p(r)$ is a continuous probability distribution function. Then $X_n$ converges weakly in distribution to a random variable $X$ which is distributed according to $\PP(X\leq r) = p(r)$.
\end{lem}

\begin{proof}[Proof of Theorem \ref{thm:fluct}]

Consider the function $f_N(x) = e^{-e^{n^{1/3}x}}$ and define $f_n^r(x) = f_n(x-r)$. Observe that this sequence of functions meets the criteria of Lemma \ref{problemma1}. Setting
\begin{equation*}
u=u(n,r,\gamma)=e^{-n\bar{f}_{\gamma} - rn^{1/3}}
\end{equation*}
observe that
\begin{equation*}
 e^{-u Z(n,n)} = f_n^r\left(\frac{\log Z(n,n) - n \bar{f}_{\gamma}}{n^{1/3}}\right).
\end{equation*}
By Lemma \ref{problemma1}, if for each $r\in \R$ we can prove that
\begin{equation*}
\lim_{n\to \infty} \ee\!\left[ f_n^r\left(\frac{\log Z(n,n) - n \bar{f}_{\gamma}}{n^{1/3}}\right) \right] = p_{\gamma}(r)
\end{equation*}
for $p_{\gamma}(r)$ a continuous probability distribution function, then it will follow that
\begin{equation*}
\lim_{n\to \infty} \pp\!\left(\frac{\log Z(n,n) - n \bar{f}_{\gamma}}{n^{1/3}}\leq r\right) = p_{\gamma}(r)
\end{equation*}
as well.

Now, the starting point of our asymptotics is the Fredholm determinant formula given in Corollary \ref{cor:COSZdet} for $\EE\!\left[ e^{-u Z(n,n)}\right]$. Observe that by setting $a_j\equiv 0$ and $\alpha_i\equiv \gamma$, we can choose $\delta_1$ and $\delta_2$ as necessary for the corollary to hold. In particular, that implies that
\begin{equation*}
\ee\!\left[ f_n^r\left(\frac{\log Z(n,n) - n \bar{f}_{\gamma}}{n^{1/3}}\right) \right] = \det(I+K_{u(n,r,\gamma)})
\end{equation*}
where $\det(I+K_{u(n,r,\gamma)})$ is the Fredholm determinant of $K\!:L^2(C_{\delta_1})\longrightarrow L^2(C_{\delta_1})$ with
\begin{equation}\label{Kunrkappa}
K_{u(n,r,\gamma)}(v,v')=\frac1{2\pi\I}\int_{\lc_{\delta_2}}dw\,\frac\pi{\sin(\pi(v-w))}
\exp\left(n\left[G(v)-G(w)\right] + r n^{1/3}(v-w)\right) \frac{dw}{w-v'},
\end{equation}
and
\begin{equation}
G(z) = \log \Gamma(z) -\log \Gamma(\gamma-z) + \bar{f}_{\gamma}z.
\end{equation}
We have rewritten the kernel in the form necessary to perform steepest descent analysis. Let us first provide a critical point derivation of the asymptotics. We will then provide the rigorous proof. Besides standard issues of estimation, we must be careful about manipulating contours due to a mine-field of poles (this ultimately leads to our present technical limitation that $\gamma<\gamma^*$).

The idea of steepest descent is to find critical points for the argument of the function in the exponential, and then to deform contours so as to go close to the critical point. The contours should be engineered such that away from the critical point, the real part of the function in the exponential decays and hence as $n$ gets large, has negligible contribution. This then justifies localizing and rescaling the integration around the critical point. The order of the first non-zero derivative (here third order) determines the rescaling in $n$ (here $n^{1/3}$) which in turn corresponds with the scale of the fluctuations in the problem we are solving. It is exactly this third order nature that accounts for the emergence of Airy functions and hence the GUE Tracy-Widom distribution.

Let us record the first three derivatives of $G(z)$:
\begin{eqnarray}
G'(z) &=& \Psi(z) + \Psi(\gamma-z) +\bar{f}_{\gamma}, \\
G''(z) &=& \Psi'(z) - \Psi'(\gamma-z),\\
G'''(z) &=& \Psi''(z) - \Psi''(\gamma-z).
\end{eqnarray}
Define $\bar{t}_{\gamma}$ such that $G''(\bar{t}_{\gamma}) = 0$; that is, $\bar{t}_{\gamma} = \gamma/2$. Notice that $\bar{f}_{\gamma}$ was defined such that $G'(\bar{t}_{\gamma}) = 0$ as well. Around $z=\bar{t}_{\gamma}$, the first non-vanishing derivative of $G$ is the third derivative. Notice that $\bar{g}_{\gamma} = -G'''(\bar{t}_{\gamma})$. This indicates that near $z=\bar{t}_{\gamma}$,
\begin{equation}
G(v)-G(w) = -\frac{\bar{g}_{\gamma} (v-\bar{t}_{\gamma})^3}{6} + \frac{\bar{g}_{\gamma}
  (w-\bar{t}_{\gamma})^3}{6}+{\rm h.o.t.},
\end{equation}
where h.o.t. denotes higher order terms in $(v-\bar{t}_{\gamma})$. This cubic behavior suggests rescaling around
$\bar{t}_{\gamma}$ by the change of variables
\begin{equation*}
\tilde{v} = n^{1/3}(v-\bar{t}_{\gamma}), \qquad \tilde{w} = n^{1/3}(w-\bar{t}_{\gamma}).
\end{equation*}
Clearly the steepest descent contour for $v$ from $\bar{t}_{\gamma}$ departs at an angle
of $\pm 2\pi /3$ whereas the contour for $w$ departs at angle $\pm \pi/3$. The $w$ contour
must lie to the right of the $v$ contour (so as to avoid the pole from $1/(w-v')$). As $n$
goes to infinity, neglecting the contribution away from the critical point, the point-wise
limit of the kernel becomes
\begin{equation}\label{Krkappa}
K_{r,\gamma}(\tilde{v},\tilde{v}') = \frac1{2\pi\I}\int \frac{1}{\tilde{v}-\tilde{w}}\frac{\exp\left\{\tfrac{-\bar{g}_{\gamma}}{6}\tilde{v}^3+ r \tilde{v}\right\}}{\exp\left\{\tfrac{-\bar{g}_{\gamma}}{6}\tilde{w}^3+ r \tilde{w}\right\}}\frac{d\tilde{w}}{\tilde{w}-\tilde{v}'}.
\end{equation}
where the kernel acts on the contour $e^{-2\pi\I/3}\rr_+\cup e^{2\pi\I/3}\rr_+$ (oriented
from negative imaginary part to positive imaginary part) and the integral in $\tilde{w}$
is on the (likewise oriented) contour $\{e^{-\pi\I/3}\rr_+ + \delta\} \cup \{e^{\pi\I/3}\rr_+ + \delta$ for any horizontal shift
$\delta>0$. This is owing to the fact that
\begin{equation*}
n^{-1/3}\frac{\pi}{\sin (\pi(v-w))} \to \frac{1}{\tilde{v}-\tilde{w}}, \qquad \frac{dw}{w-v'} \to \frac{d\tilde{w}}{\tilde{w}-\tilde{v}'}.
\end{equation*}
where the $n^{-1/3}$ came from the Jacobian associated with the change of variables in $v$ and $v'$.

Another change of variables to rescale by $(\tilde{g}_{\gamma}/2)^{1/3}$ results in
\begin{equation*}
K_{r,\gamma}(\tilde{v},\tilde{v}') = \frac1{2\pi\I}\int \frac{1}{\tilde{v}-\tilde{w}}\frac{\exp\left\{-\tilde{v}^3 /3+ (\bar{g}_{\gamma}/2)^{-1/3} r \tilde{v}\right\}}{\exp\left\{-\tilde{w}^3 /3 + (\bar{g}_{\gamma}/2)^{-1/3} r \tilde{w}\right\}}\frac{d\tilde{w}}{\tilde{w}-\tilde{v}'}.
\end{equation*}
One now recognizes that the Fredholm determinant of this kernel is one way to define the
GUE Tracy-Widom distribution (see for instance Lemma 8.6 in \cite{borodinCorwinFerrari}). This shows that the limiting expectation 
$p_{\gamma}(r)=
F_{{\rm GUE}}\left((\bar{g}_{\gamma} / 2)^{-1/3}r\right)$ which shows that it is a
continuous probability distribution function and thus Lemma \ref{problemma1} applies. This
completes the critical point derivation.

The challenge now is to rigorously prove that $\det(I+K_{u(n,r,\gamma)})$ converges to
$\det(I+K_{r,\gamma})$ where the two operators act on their respective $L^2$ spaces and
are defined with respect to the kernels above in \eqref{Kunrkappa} and \eqref{Krkappa}.

We will consider the case where $\gamma<\gamma^*$ for $\gamma^*$ small enough and perform
certain estimates given that assumption. Let us record some useful estimates: for $z$
close to zero we have ($g$ is being used to represent the Euler Mascheroni constant $g
=0.577...$)
\begin{eqnarray*}
\log\Gamma(z) &=& -\log z -g z + O(z^2)\\
\Psi(z) &=& -\frac{1}{z} -g + O(z)\\
\Psi'(z) &=& \frac{1}{z^2} + O(1)\\
\Psi''(z) &=& -\frac{2}{z^3} + O(1)\\
\Psi'''(z) &=& \frac{6}{z^4} + O(1).
\end{eqnarray*}
From this we can immediately observe that for $\gamma$ small,
\begin{eqnarray}
\bar{t}_{\gamma} &=& \gamma/2\\
\bar{f}_{\gamma} &=& 4\gamma^{-1} + 2g + O(\gamma)\\
\bar{g}_{\gamma} &=& 32 \gamma^{-3} + O(1).
\end{eqnarray}

With the change of variables $z= \gamma\tilde{z}$ we may estimate
\begin{equation}
  \begin{aligned}
    G(z)-G(\bar{t}_{\gamma}) &=\log\Gamma(z) - \log\Gamma(\gamma-z) + \bar{f}_{\gamma} (z -\gamma/2)
    +\kappa \bar{t}_{\kappa}^2/2 - \bar{f}_{\kappa} \bar{t}_{\kappa}\\
    &= f(\tilde{z}) + O(\gamma^2),
  \end{aligned}    \label{festimate}
\end{equation}
where the error is uniform for $\tilde{z}$ in any compact domain and
\begin{equation*}
f(\tilde{z}) = \log(1-\tilde z) - \log( \tilde z) + 4 \tilde z -2.
\end{equation*}

Our approach will be as follows: Step 1: We will deform the contour $C_{\delta_1}$ on
which $v$ and $v'$ are integrated as well as the contour on which $w$ is integrated so
that they both locally follow the steepest descent curve for $G(z)$ coming from the
critical point $\bar{t}_{\gamma}$ and so that along them there is sufficient global decay
to ensure that the integral localizes to the vicinity of $\bar{t}_{\gamma}$. Step 2: In
order to show this desired localization we will use a Taylor series with remainder
estimate in a small ball of radius approximately $\gamma$ around $\bar{t}_{\gamma}$, and
outside that ball we will use the estimate \eqref{festimate} for the $v$ and $v'$ contour,
and a similarly straightforward estimate for the $w$ contour. Step 3: Given these
estimates we can show convergence of the Fredholm determinants as desired.

\noindent {\bf Step 1:} Define the contour $C_{f}$ which corresponds to a steep descent
contour\footnote{For an integral $I=\int_\gamma dz\, e^{t f(z)}$, we say that $\gamma$ is
  a steep descent path if (1) $\Real(f(z))$ reaches the maximum at some $z_0\in\gamma$:
  $\Real(f(z))< \Real(f(z_0))$ for $z\in\gamma\setminus\{z_0\}$, and (2) $\Real(f(z))$ is
  monotone along $\gamma$ except at its maximum point $z_0$ and, if $\gamma$ is closed, at
  a point $z_1$ where the minimum of $\Real(f)$ is reached.} of the function $f(\tilde{z})$ leaving $\tilde{z}=1/2$
at an angle of $2\pi/3$ and returning to $\tilde{z}=1/2$ at an angle of $-2\pi/3$, given a
positive orientation. In particular we will take $C_{f}$ to be composed of a line segment
from $\tilde{z}=1/2$ to $\tilde{z}=1/2 + e^{ 2\pi\I/3}$, then a circular arc (centered at
0) going counter-clockwise until $\tilde{z}=1/2+e^{-2\pi\I/3}$ and then a line segment
back to $\tilde{z}=1/2$. It is elementary to confirm that along $C_f$, $\Real(f)$ achieves
its maximal value of 0 at the unique point $\tilde{z}=1/2$ and is negative everywhere
else. The contour $\gamma C_{f}$ will serve as our steep descent contour for $G(v)$ (see
Figure \ref{contour1}).

One should note that for any $r\in (0,1)$, for $\gamma$ small enough, the contour $\gamma C_{f}$ is such that for $v$ and $v'$ along it, $|v-v'|<r$. By virtue of this fact, we may employ Proposition \ref{TWprop1} to deform our initial contour $C_{\delta_1}$ to the contour $\gamma C_{f}$ without changing the value of the Fredholm determinant.

Likewise, we must deform the contour along which the $w$ integration is performed. Given that $v,v'\in \gamma C_{f}$ now, we may again use Proposition \ref{TWprop1} to deform the $w$ integration to a contour $C_{\langle,n}$ which is defined symmetrically over the real axis by a ray from $\bar{t}_{\gamma}+n^{-1/3}$ leaving at an angle of $\pi /3$ (again, see Figure \ref{contour1}). It is easy to see that this deformation does not pass any poles, and with ease one confirms that due to the linear term in $G(w)$ the deformation of the infinite contour is justified by suitable decay bounds near infinity between the original and final contours.

Thus, the outcome of this first step is that the $v,v'$ contour is now given by $\gamma
C_{f}$ and the $w$ contour is now given by $C_{\langle,n}$ which is independent of $v$.

\begin{figure}
\begin{center}
\includegraphics[width=.5\textwidth,height=.5\textwidth]{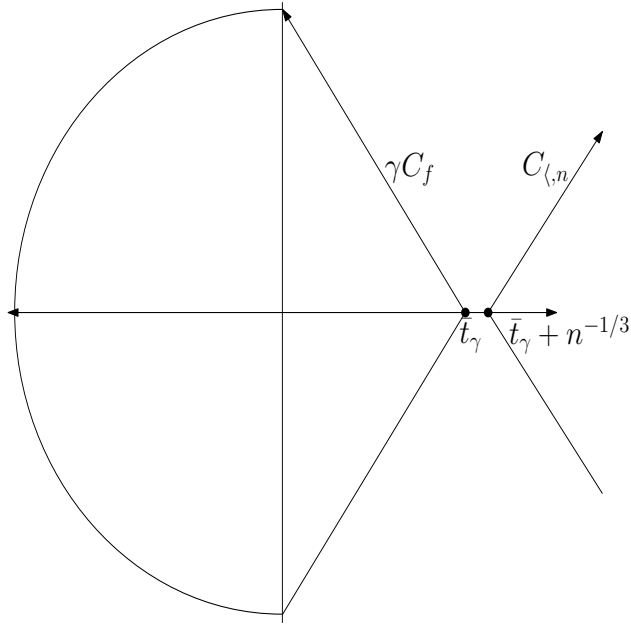}
\end{center}
\caption{Steep descent contours}\label{contour1}
\end{figure}

\noindent {\bf Step 2:} We will presently provide two types of estimates for our function $G$ along the specified contours: those valid in a small ball around the critical point and those valid outside the small ball. Let us first focus on the small ball estimate.

\begin{lem}\label{lem:smallball}
There exists $\gamma^*>0$ such that for all $\gamma<\gamma^*$ the following two facts hold:

\noindent(1) There exists a constant $c_1>0$ such that for all $v$ along the straight line segments of $\gamma C_{f}$:
\begin{equation*}
\Real\left[G(v)-G(\bar{t}_{\gamma})\right] \leq \Real\left[ -c_1\gamma^3 (v-\bar{t}_{\gamma})^3 \right].
\end{equation*}

\noindent(2) There exists a constant $c_2>0$ such that for all $w$ along the contour
$C_{\langle,n}$ at distance less than $\gamma$ from $\bar{t}_{\gamma}$:
\begin{equation*}
\Real\left[G(w)-G(\bar{t}_{\gamma})\right] \geq \Real\left[ -c_2\gamma^3 (w-\bar{t}_{\gamma})^3 \right].
\end{equation*}
\end{lem}
\begin{proof}
Recall the Taylor expansion remainder estimate for a function $F(z)$ expanded around $\bar{z}$,
\begin{equation}
  \left| F(z) - \left( F(\bar{z}) + F'(\bar{z})(z-\bar{z})+ \tfrac{1}{2} F''(\bar{z})(z-\bar{z})^2 + \tfrac{1}{6} F'''(\bar{z})(z-\bar{z})^3 \right)\right|
 \leq \max_{\zeta \in B(\bar{z},|z-\bar{z}|)} \tfrac{1}{24}|F''''(\zeta)||z-\bar{z}|^4.
\end{equation}
We may apply this to our function $G(z)$ around the point $\bar{t}_{\gamma}$ giving
\begin{equation*}
\left| G(z) - G(\bar{t}_{\gamma}) + \tfrac{1}{6} \bar{g}_{\gamma}(z-\bar{t}_{\gamma})^3\right| \leq \max_{\zeta\in B(\bar{t}_{\gamma},|z-\bar{t}_{\gamma}|)} \tfrac{1}{24}|G''''(\zeta)||z-\bar{t}_{\gamma}|^4.
\end{equation*}
Let $z=\gamma\tilde{z}$ and also let $\zeta=\gamma\tilde{\zeta}$. Then for $\gamma$ small, we have the following estimate
\begin{equation*}
\left| G(z) - G(\bar{t}_{\gamma}) + \tfrac{16}{3}(\tilde{z}-1/2)^3\right| \leq \frac{1}{4} \left|\frac{1}{\tilde z^3} - \frac{1}{(1-\tilde z)^4}\right| \left|\tilde z - 1/2 \right|^4 + O(\gamma^2).
\end{equation*}
Both parts of the lemma follow readily from this estimate and the comparison along the contours of interest of $\tfrac{16}{3}(\tilde{z}-1/2)^3$ with the right hand side above.
\end{proof}

We may now turn to the estimate outside the ball of size $\gamma$.
\begin{lem}\label{lem:circ}
There exists $\gamma^*>0$ and $c>0$ such that for all $v$ along the circular part of the contour $\gamma C_{f}$, the following holds:
\begin{equation*}
\Real\!\left[G(v)-G(\bar{t}_{\gamma})\right] \leq -c.
\end{equation*}
\end{lem}
\begin{proof}
Writing $v=\gamma\tilde{v}$ we may appeal to the second line of the estimate
\eqref{festimate} and the fact that along the circular part of the contour $\gamma C_{f}$,
the real part of function $f(\tilde v)$ is strictly negative and the error in the estimate is $O(\gamma)$.
\end{proof}

The above bound suffices for the $v$ contour since it is finite length. However, the $w$ contour is infinite so our estimate must be sufficient to ensure that the contour's tails do not play a role.
\begin{lem}\label{lem:line}
  There exists $\gamma^*>0$ and $c>0$ such that for all $w$ along the contour
  $C_{\langle,n}$ at distance exceeding $\gamma$ from $\bar{t}_{\gamma}$, the following
  holds:
\begin{equation*}
\Real\!\left[G(w)-G(\bar{t}_{\gamma})\right] \geq \Real\!\left[c\gamma^{-1} w\right].
\end{equation*}
\end{lem}
\begin{proof}
  This estimate is best established in three parts. We first estimate for $\zeta$ between
  distance $\gamma$ and distance $c_1 \gamma$ from $\bar{t}_{\gamma}$ ($c_1$
  large). Second we estimate for $w$ between distance $c_1 \gamma $ and distance $c_2$
  from $\bar{t}_{\gamma}$. Finally we estimate for all $w$ yet further. This third
  estimate is immediate from the first line of \eqref{festimate} in which the linear term
  in $z$ clearly exceeds the other terms for $\gamma$ small enough and $|z|>c_2$.

  To make the first estimate we use the bottom line of \eqref{festimate}. The function
  $f(\tilde{z})$ has sufficient growth along this contour to overwhelm the $O (\gamma^2)$
  error as long as $\gamma$ is small enough. The $O(\gamma^2)$ error in \eqref{festimate}
  is only valid for $\tilde{z}$ in a compact domain though. So, for the second estimate we
  must use the cruder bound that $G(z)-G(\bar{t}_{\gamma}) = f(\tilde{z}) +O(1)$ for $z$
  along the contour and of size less than $c_2$ from $\bar{t}_{\gamma}$. Since in this
  regime of $z$, $f(\tilde{z})$ behaves like $(z -\gamma/2)\bar{f}_{\gamma} = 4\gamma^{-1}
  z +O(1)$, one sees that for $\gamma$ small enough, the $O(1)$ error is overwhelmed and
  the claimed estimate follows.
\end{proof}

\noindent{\bf Step 3:} We now employ the estimates given above to conclude that
$\det(I+K_{u(n,r,\gamma)})$ converges to $\det(I+K_{r,\gamma})$ where the two operators
act on their respective $L^2$ spaces and are defined with respect to the kernels above in
\eqref{Kunrkappa} and \eqref{Krkappa}. The approach is standard (see for instance \cite{borodinCorwinFerrari, ACQ}) so we just briefly review
what is done. Convergence can either be shown at the level of Fredholm series expansion or
trace-class convergence of the operators. Focusing on the Fredholm series expansion, the
estimates provided by Lemmas \ref{lem:smallball}, \ref{lem:circ} and \ref{lem:line}, along
with Hadamard's bound, show that for any $\e$ there is a $k$ large enough such that for
all $n$, the contribution of the terms of the Fredholm series expansion past index $k$ can
be bounded by $\e$. This localizes the problem of asymptotics to a finite number of
integrals involving the kernels. The same estimates then show that these integrals can be
localized in a large window of size $n^{-1/3}$ around the critical point
$\bar{t}_{\gamma}$. The cost of throwing away the portion of the integrals outside this
window can be uniformly (in $n$) bounded by $\e$, assuming the window is large
enough. Finally, after a change of variables to rescale this window by $n^{1/3}$ we can
use the Taylor series with remainder to show that as $n$ goes to infinity, the integrals
coming from the kernel $K_{u(n,r,\gamma)}$ converge to those coming from
$K_{r,\gamma}$. This last step is essentially the content of the critical point
computation given earlier. 
\end{proof}

\section{Equivalence of contour integral and Fredholm determinants}\label{Sec:equiv}

Let us first recall a version of the Andr\'{e}ief identity \cite{And} (sometimes referred
to as the generalized Cauchy-Binet identity): Assuming all integrals exist,
\begin{equation}
\frac{1}{N!} \int \cdots \int d\mu(z_1)\cdots d\mu(z_N)
\det\left[f_i(z_j)\right]_{i,j=1}^{N} \det\left[g_i(z_j)\right]_{i,j=1}^{N}
= \det\left[\int d\mu(z)\, f_i(z)g_j(z) \right]_{i,j=1}^{N}.
\label{CauchyBinet}
\end{equation}

\begin{proof}[Proof of Theorem \ref{thm:equiv}]
  We start working from the right hand side of the formula. We regard the Fredhom
  determinant as given by its Fredholm series expansion (the convergence of the series is
  a consequence of the identity we will prove). Note that our assumptions on $\delta_1$
  and $\delta_2$ imply that $\delta_1<\frac12$, so the factors $\Gamma(v-a_m)$ have no
  poles for $v\in C_{\delta_1}$. They also imply that the factor $\sin(\pi(v-w))(w-v')$ in
  the denominator in \eqref{eq:defK} is never zero on $\lc_{\delta_2}$. To see that the
  integral is convergent, recall the asymptotics $|\Gamma(x+\I
  y)|e^{\pi/2|y|}|y|^{1/2-x}\to\sqrt{2\pi}$ as $y\to\pm\infty$ for $x,y\in\rr$ (see
  (6.1.45) in \cite{abrSteg}), which implies that, for $\Real(z)$ in some finite interval,
  \begin{equation}
    \label{eq:gammaBd}
    |\Gamma(z)|\geq c_1e^{-\frac{\pi}{2}|\!\im(z)|}|\!\im(z)|^{\eta}
  \end{equation}
  as $|\!\im(z)|\to\infty$ for some $c_1>0$ and $\eta\in\rr$. We also have, under the
  same assumption on $z$, $|\!\sin(\pi z)|\geq c_2e^{\pi|\!\im(z)|}$ for some $c_2>0$.
  Then
  \begin{equation}
    \left|\frac\pi{\sin(\pi(v-w))}\frac{F(w)}{F(v)}\frac1{w-v'}\right|\prod_{m=1}^N\frac{\Gamma(v-a_m)}{\Gamma(w-a_m)}
    \leq c_3e^{\pi(\frac{N}{2}-1)|\!\im(w)|}|\!\im(w)|^{-N\eta-1}|F(w)|\label{eq:decay}
  \end{equation}
  as $|\!\im(w)|\to\infty$ for some $c_3>0$, and hence the integral in \eqref{eq:defK} is convergent by \eqref{eq:decayAssum}.


  Observe that $K=AB$ with $A\!:L^2(\lc_{\delta_2})\longrightarrow L^2(C_{\delta_1})$ and
  $B\!:L^2(C_{\delta_1})\longrightarrow L^2(\lc_{\delta_2})$ given by
  \[A(v,w)=\frac\pi{\sin(\pi(v-w))}\frac{F(w)}{F(v)}\prod_{m=1}^N\frac{\Gamma(v-a_m)}{\Gamma(w-a_m)}
  \qqand B(w,v')=\frac1{w-v'}\] (checking that $A$ and $B$ map $L^2(\lc_{\delta_2})$ to
  $L^2(C_{\delta_1})$ and vice-versa is easy, and uses \eqref{eq:decayAssum} in the case
  of $A$). Let $\wt K=BA\!:L^2(\lc_{\delta_2})\longrightarrow L^2(\lc_{\delta_2})$, which
  is then given by
  \[\wt
  K(w,w')=\frac1{2\pi\I}\oint_{C_{\delta_1}}dv\,\frac{1}{w-v}\frac\pi{\sin(\pi(v-w'))}\frac{F(w')}{F(v)}\prod_{m=1}^N\frac{\Gamma(v-a_m)}{\Gamma(w'-a_m)}.\]
  Now recall in general that if $A$ and $B$ have integral kernels so that the integrals
  $\int dw\,A(v,w)B(w,v')$ and $\int dv\,B(w,v)A(v,w')$ both converge for $v,v',w,w'$ in
  the right domains, then $\det(I+AB)=\det(I+BA)$ in the sense that the two convergent Fredholm
  expansions coincide termwise (this can be seen as an application of the Andr\'{e}ief
  identity \eqref{CauchyBinet}). Using this we deduce that
  \begin{equation}
    \label{eq:detKwtK}
    \det(I+K)_{L^2(C_{\delta_1})}=\det(I+\wt K)_{L^2(\lc_{\delta_2})}.
  \end{equation}

  Recalling that $\Gamma(1+s)=s\Gamma(s)$, let
  \begin{equation}
    G(v)=F(v)\prod_{m=1}^N\frac{v-a_m}{\Gamma(v-a_m+1)}\label{eq:defG}
  \end{equation}
  and rewrite $\wt K$ as
  \[\wt K(w,w')=\frac1{2\pi\I}\oint_{C_{\delta_1}}dv\,\frac{1}{w-v}\frac\pi{\sin(\pi(v-w'))}\frac{G(w')}{G(v)}.\]
  By the assumption $\delta_1<\min\{\delta_2,1-\delta_2\}$, $w-v$ and $\sin(\pi(v-w'))$
  are never zero for $v\in C_{\delta_1}$ and $w,w'\in \lc_{\delta_2}$. Thus the only poles
  of the integrand inside $C_{\delta_1}$ are those coming from $G(v)^{-1}$, which can only
  come from the factors $(v-a_m)$ in the definition of $G(v)$. For simplicity we will assume
  that all the $a_i$'s are pairwise distinct, the general case follows likewise by taking limits.
  Then the integrand has simple poles at $v=a_i$, $i=1,\dotsc,N$, and thus
  \begin{align*}
    \wt K(w,w')&=\frac1{2\pi\I}\sum_{i=1}^NG(w')\Res{v=a_i}\!\left[\frac{1}{w-v}\frac\pi{\sin(\pi(v-w'))}\frac1{G(v)}\right]\\
    &=\frac1{2\pi\I}\sum_{i=1}^NG(w')\frac{1}{w-a_i}\frac\pi{\sin(\pi(a_i-w'))}\frac1{F(a_i)}\prod_{m\neq i}\frac{\Gamma(a_i-a_m+1)}{a_i-a_m}.
  \end{align*}
  We rewrite this as
  \[\wt K(w,w')=\frac1{2\pi\I}\sum_{i=1}^Nf_i(w)g_i(w')\]
  with
  \begin{equation}
    f_i(w)=\frac{1}{w-a_i},\quad g_i(w')=C_i G(w')\frac\pi{\sin(\pi(a_i-w'))}\qand
    C_i=\frac1{F(a_i)}\prod_{m\neq i}\Gamma(a_i-a_m).\label{eq:deffgalpha}
  \end{equation}
  In particular, this means that $\wt K=\wt A\wt B$ where $\wt A\!:\ell^2(\{1,\dotsc,N\})\longrightarrow
  L^2(\lc_{\delta_2})$ is given by $\wt A(w,i)=(2\pi\I)^{-1}f_i(w)$ and
  $\wt B\!:L^2(\lc_{\delta_2})\longrightarrow\ell^2(\{1,\dotsc,N\})$ by $\wt B(i,w')=g_i(w')$, so that
  \[\wt B\wt A(i,j)=\frac1{2\pi\I}\int_{\lc_{\delta_2}}dw\,f_i(w)g_i(w).\]
  The integral is convergent for the same reason as the integral in \eqref{eq:defK},
  and hence using again the formula $\det(I+AB)=\det(I+BA)$ we deduce that
  \begin{equation}
    \label{eq:wtKwtAwtB}
    \det(I+\wt K)_{L^2(\lc_{\delta_2})}=\det\!\left[\uno{i=j}+\frac1{2\pi\I}\int_{\lc_{\delta_2}}dw\,f_i(w)g_j(w)\right]_{i,j=1}^N.
  \end{equation}

  Our next goal is to shift the contour $\lc_{\delta_2}$ in the integral appearing on the
  right hand side of \eqref{eq:wtKwtAwtB} to $-\lc_{\delta_1}$. Note that, as we do this, we
  will cross all the $a_i$'s. We have
  \[f_i(w)g_j(w)=F(w)\frac{C_j\pi}{\sin(\pi(a_j-w))}\frac1{\Gamma(w-a_i+1)}
  \prod_{m\neq i}\frac1{\Gamma(w-a_m)}.\] Since all the poles of $F$ lie to the right of
  $\lc_{\delta_2}$, the only singularities of $f_i(w)g_j(w)$ we encounter as we shift the
  contour are those coming from the zeros of the sine, and hence we only see simple poles
  at $w=a_i$ in the case $i=j$. On the other hand, the integral of $f_i(w)g_j(w)$ along
  the segments going from $-\delta_1\pm\I M$ to $\delta_2\pm\I M$ goes to 0 as
  $|M|\to\infty$ by \eqref{eq:decayAssum} and \eqref{eq:decay} as before. The conclusion
  is that
  \[\frac1{2\pi\I}\int_{\lc_{\delta_2}}dw\,f_i(w)g_j(w)
  =\uno{i=j}\Res{w=a_i}\!\left[f_i(w)g_i(w)\right]+\frac1{2\pi\I}\int_{-\lc_{\delta_1}}dw\,f_i(w)g_j(w).\]
  Since
  \[\Res{w=a_i}\!\left[f_i(w)g_i(w)\right]=-C_iF(a_i)\prod_{m\neq
    i}\frac{a_i-a_m}{\Gamma(a_i-a_m+1)}=-1,\]
  we deduce from \eqref{eq:wtKwtAwtB} that
  \begin{equation}
    \label{eq:newDetwtK}
    \begin{split}
      \det(I+\wt K)_{L^2(\lc_{\delta_1})}&=\det\!\left[\frac1{2\pi\I}\int_{-\lc_{\delta_1}}dw\,f_i(w)g_j(w)\right]_{i,j=1}^N\\
      &=\frac1{(2\pi\I)^NN!}\int_{-\lc_{\delta_1}}\!\dotsm\int_{-\lc_{\delta_1}}dw_1\dotsm
      dw_N\,\det[f_i(w_j)]_{i,j=1}^N\det[g_i(w_j)]_{i,j=1}^N,
    \end{split}
  \end{equation}
  where the second equality follows from the Andr\'{e}ief identity \eqref{CauchyBinet}.

  Note that the contours in this last integral are exactly the ones appearing in the
  formula we seek to prove. Hence what is left is to compute the two determinants and
  show that the integrand coincides with the one appearing in \eqref{eq:equiv}. We start
  by observing that $\det[f_i(w_j)]_{i,j=1}^N$ is a Cauchy determinant, so
  \begin{equation}
    \label{eq:fiwj}
    \det[f_i(w_j)]_{i,j=1}^N=\frac{\prod_{i<j}(a_j-a_i)(w_i-w_j)}{\prod_{i,j=1}^N(w_i-a_j)}.
  \end{equation}
  On the other hand
  \[\det[g_i(w_j)]_{i,j=1}^N=\prod_{i=1}^NG(w_i)C_i\det\!\left[\frac\pi{\sin(\pi(a_i-w_j))}\right]_{i,j=1}^N.\]
  The determinant on the right hand side can be turned into another Cauchy determinant by
  writing
  \[\frac\pi{\sin(\pi(a_i-w_j))}=\frac{2\pi\I}{e^{-\I\pi(a_i+w_j)}}\frac1{e^{2\I\pi
      a_i}-e^{2\I\pi w_j}},\]
  so that
  \[\det\!\left[\frac\pi{\sin(\pi(a_i-w_j))}\right]_{i,j=1}^N=\prod_{i=1}^N\frac{2\pi\I}{e^{-\I\pi(a_i+w_i)}}
    \det\!\left[\frac1{e^{2\I\pi a_i}-e^{2\I\pi w_j}}\right]_{i,j=1}^N,\]
 and therefore
 \begin{equation}
   \label{eq:giwj}
   \det[g_i(w_j)]_{i,j=1}^N=(2\pi\I)^Ne^{\I\pi\sum_i(a_i+w_i)}\prod_{i=1}^NG(w_i)C_i
   \frac{\prod_{i<j}(e^{2\I\pi a_j}-e^{2\I\pi a_i})(e^{2\I\pi w_i}-e^{2\I\pi w_j})}{\prod_{i,j=1}^N(e^{2\I\pi a_i}-e^{2\I\pi w_j})}.
 \end{equation}
 Now we collect the factors involving only $w_i$'s, only $a_i$'s, and cross-terms: from
 \eqref{eq:fiwj}, \eqref{eq:giwj} and the definitions \eqref{eq:defG} and
 \eqref{eq:deffgalpha} of $G(w_i)$ and $C_i$ we get
 \begin{equation}
   \label{eq:detdet}
   \det[f_i(w_j)]_{i,j=1}^N\det[g_i(w_j)]_{i,j=1}^N=(2\pi\I)^ND_{a,a}D_{w,w}D_{a,w}
 \end{equation}
 with
 \begin{align}
   D_{a,a}&=e^{\I\pi\sum_ia_i}\prod_{i<j}(a_j-a_i)(e^{2\I\pi a_i}-e^{2\I\pi
     a_j})\prod_{i=1}^N\frac{\prod_{m\neq i}\Gamma(a_i-a_m)}{F(a_i)},\\
   D_{w,w}&=e^{\I\pi\sum_iw_i}\prod_{i=1}^NF(w_i)\prod_{i<j}(w_i-w_j)(e^{2\I\pi w_i}-e^{2\I\pi
     w_j}),\\
   D_{a,w}&=\prod_{i,j=1}^N\frac1{w_i-a_j}\frac1{\Gamma(w_i-a_j)}\frac1{e^{2\I\pi a_i}-e^{2\I\pi
     w_j}}.
 \end{align}
$D_{a,a}$ can be rearranged in the following way:
\begin{align}
  D_{a,a}&=e^{\I\pi\sum_ia_i}\prod_{i<j}(a_j-a_i)e^{\I\pi(a_i+a_j)}(e^{\I\pi(a_i-a_j)}-e^{\I\pi
     (a_j-a_i)})\frac{\prod_{i\neq j}\Gamma(a_i-a_j)}{\prod_{i=1}^NF(a_i)}\\
   &=e^{\I\pi N\sum_ia_i}(2\I)^{\frac12N(N-1)}\prod_{i<j}\Gamma(a_j-a_i+1)\Gamma(a_i-a_j)\sin(\pi(a_j-a_i))\prod_{i=1}^NF(a_i)^{-1}\\
   &=e^{\I\pi N\sum_ia_i}(-2\pi\I)^{\frac12N(N-1)}\prod_{i=1}^NF(a_i)^{-1},
\end{align}
where have used
\begin{equation*}
 \Gamma(-s)\Gamma(1+s)=\frac\pi{\sin(-\pi s)}.
\end{equation*}
Similar computations give
\begin{align}
  D_{w,w}&=e^{\I\pi
    N\sum_iw_i}(2\pi\I)^{\frac12N(N-1)}\prod_{i<j}\frac{\sin(\pi(w_i-w_j))}{\pi}(w_i-w_j)\prod_{i=1}^NF(w_i),\\
  D_{a,w}&=e^{-\I\pi N\sum_i(a_i+w_i)}(2\pi\I)^{-N^2}\prod_{i,j=1}^N\Gamma(a_j-w_i).
\end{align}
Hence
\begin{align}
  D_{a,a}D_{w,w}D_{a,w}&=\frac{1}{(2\pi\I)^N}(-1)^{\frac12N(N-1)}\prod_{i<j}\frac{\sin(\pi(w_i-w_j))}{\pi}(w_i-w_j)
  \prod_{i,j=1}^N\Gamma(a_j-w_i)\prod_{i=1}^N\frac{F(w_i)}{F(a_i)}\\
  &=\frac{1}{(2\pi\I)^N}\prod_{i\neq j}\frac1{\Gamma(w_i-w_j)}\prod_{i=1}^N\frac{F(w_i)}{F(a_i)},
\end{align}
and thus from \eqref{eq:newDetwtK} and \eqref{eq:detdet} we deduce that
\[\det(I+\wt K)_{L^2(\lc_{\delta_1})}
=\frac1{(2\pi\I)^NN!}\int_{-\lc_{\delta_1}}\!\dotsm\int_{-\lc_{\delta_1}}dw_1\dotsm
dw_N\,\prod_{i\neq
  j}\frac1{\Gamma(w_i-w_j)}\prod_{i,j=1}^N\Gamma(a_j-w_i)\prod_{i=1}^N\frac{F(w_i)}{F(a_i)}.\]
Putting this together with \eqref{eq:detKwtK} and the definition \eqref{eq:skly} of $s_N$
yields \eqref{eq:equiv}.
\end{proof}



\begin{thebibliography}{alpha}
\bibitem{abrSteg}
M. Abramowitz and I.~A. Stegun.
\newblock {\em Handbook of mathematical functions with formulas, graphs, and
  mathematical tables}, volume~55.
\newblock National Bureau of Standards Applied Mathematics Series, 1964.

\bibitem{AKQ}
T.~Alberts, K.~Khanin, J.~Quastel.
\newblock Intermediate disorder regime for 1+1 dimensional directed polymers.
\newblock arXiv:1202.4398 (2012).


\bibitem{ACQ}
G.~Amir, I.~Corwin, J.~Quastel.
\newblock Probability distribution of the free energy of the continuum directed random polymer in $1+1$ dimensions.
\newblock {\em Comm. Pure Appl. Math.}, {\bf 64}:466--537, 2011.

\bibitem{And}
C. Andr\'{e}ief.
\newblock Note sur une relation entre les int\'{e}grales d\'{e}finies des produits des fonctions.
\newblock  {\em M\'{e}m. de la Soc. Sci., Bordeaux} {\bf 2}, 1--14, 1883.

\bibitem{ABC}
A.~Auffinger, J.~Baik, I.~Corwin.
\newblock Universality for directed polymers in thin rectangles.
\newblock arXiv:1204.4445 (2012).

\bibitem{BBP}
J.~ Baik, G.~Ben Arous, S.~Pech\'{e}.
\newblock Phase transition of the largest eigenvalue for non-null complex sample covariance matrices.
\newblock {\em Ann. Probab.}, {\bf 33}:1643--1697, 2006.

\bibitem{borCor}
A. Borodin and I. Corwin.
\newblock Macdonald processes.
\newblock arXiv:1111.4408 (2011).

\bibitem{borodinCorwinFerrari}
A. Borodin, I. Corwin, and P. Ferrari.
\newblock Free energy fluctuations for directed polymers in random media in 1+1
  dimension.
\newblock arXiv:1204.1024 (2012).

\bibitem{BP07}
A.~Borodin, S.~P\'ech\'e.
\newblock Airy kernel with two sets of parameters in directed percolation and random matrix theory.
\newblock {\em J. Stat. Phys.}, {\bf 132}:275--290, 2008.

\bibitem{corwinReview}
I. Corwin.
\newblock The {K}ardar-{P}arisi-{Z}hang equation and universality class.
\newblock {\em Random Matrices Theory Appl.}, {\bf 1}, 2012.

\bibitem{cosz}
I. Corwin, N. O'Connell, T. Sepp\"al\"ainen, and N. Zygouras.
\newblock Tropical combinatorics and {W}hittaker functions.
\newblock arXiv:1110.3489 (2011).

\bibitem{CQ}
I.~Corwin, J.~Quastel.
\newblock Universal distribution of fluctuations at the edge of the rarefaction fan.
\newblock {\em Ann. Probab.}, to appear (arXiv:1006.1338).

\bibitem{QRMF}
G.~Moreno Flores, J.~Quastel,  D.~Remenik.
\newblock In preparation.

\bibitem{oconnellQTL}
N. O'Connell.
\newblock Directed polymers and the quantum {T}oda lattice.
\newblock {\em Ann. Probab.}, {\bf 40}:437--458, 2012.

\bibitem{seppPolymerBoundary}
T. Sepp\"al\"ainen.
\newblock Scaling for a one-dimensional directed polymer with boundary.
\newblock {\em Ann. Probab.}, {\bf 40}:19--73, 2012.

\bibitem{tracyWidom}
C.~Tracy, H.~Widom.
\newblock Level-spacing distributions and the {A}iry kernel.
\newblock {\em Comm. Math. Phys.}, {\bf 159}:151--174, 1994.

\bibitem{TW3}
C.~Tracy, H.~Widom.
\newblock Asymptotics in ASEP with step initial condition.
\newblock {\em Comm. Math. Phys.}, {\bf 290}:129--154, 2009.

\end{thebibliography}
\end{document}